\documentclass[twoside,leqno,
]{amsart}
\usepackage{amssymb,amsmath,amsthm,latexsym,stmaryrd,xcolor}

\newtheorem{thm}{Theorem}[section]
 \newtheorem{cor}[thm]{Corollary}
 
 \newtheorem{prop}[thm]{Proposition}
 \theoremstyle{definition}
 
 \theoremstyle{remark}

 \numberwithin{equation}{section}

\keywords{Almost paracontact structure, almost paracomplex structure, Riemannian metric, cone, hyperbolic extension}
\subjclass[2010]{Primary 53C15, 53C17; Secondary 53C25, 53C42}

\newcommand{\ie}{i.e.\ }
\newcommand{\f}{\phi}
\newcommand{\tg}{\tilde{g}}
\newcommand{\n}{\nabla}

\newcommand{\M}{(\mathcal{M},\allowbreak{}\f,\allowbreak{}\xi,\allowbreak{}\eta,g)}

\newcommand{\I}{\mathcal{I}}
\newcommand{\W}{\mathcal{W}}
\newcommand{\R}{\mathbb R}

\newcommand{\X}{\mathfrak X}
\newcommand{\F}{\mathcal{F}}

\newcommand{\HH}{\mathcal{H}}

\newcommand{\MM}{\mathcal{M}}
\newcommand{\NN}{\mathcal{N}}

\newcommand{\tr}{{\rm tr}}
\newcommand{\ta}{\theta}

\newcommand{\om}{\omega}
\newcommand{\lm}{\lambda}

\newcommand{\al}{\alpha}

\newcommand{\D}{\mathrm{d}}

\newcommand{\ddt}{\tfrac{\D}{\D t}}
\newcommand{\A}{\allowbreak{}}
\newcommand{\s}{\sharp}

\newcommand{\thmref}[1]{Theorem~\ref{#1}}

\title
[Almost paracontact almost paracomplex Riemannian manifolds as \ldots ]
{Almost paracontact almost paracomplex Riemannian manifolds as extensions \\ of 2-dimensional space-forms
}

\author[M. Manev]{Mancho  Manev}
\author[V. Tavkova]{Veselina  Tavkova}
\address[MM1, VT]{Department of Algebra and Geometry, Faculty of Mathematics and Informatics,
University of Plovdiv Paisii Hilendarski, 24, Tzar Asen St, 4000 Plovdiv,
Bulgaria}
\email{mmanev@uni-plovdiv.bg}
\email{vtavkova@uni-plovdiv.bg}

\address[MM2]{
Department of Medical Informatics, Biostatistics and E-Learning,
Faculty of Public Health, Medical University of Plovdiv,
15A, Vasil Aprilov Blvd, 4002 Plovdiv,
Bulgaria
}

\begin{document}

%


\begin{abstract}
Almost pa\-ra\-con\-tact Riemannian manifolds of the lowest dimension are studied, whose paracontact distributions are equipped with an almost paracomplex structure.
These manifolds are constructed as a product of a real line and a 2-dimensional Riemannian space-form.
Their metric is obtained in two ways: as a cone metric and as a hyperbolic extension of the metric of the underlying paracomplex 2-manifold. The resulting manifolds are studied and characterised in terms of the
classification used and their curvature properties.

\end{abstract}

\maketitle

\section*{\textbf{Introduction}}\label{sec-intro}
 \vglue-10pt
 \indent

In 1976, I.\,Sato \cite{Sato76}  introduced  the concept of an almost paracontact structure compatible with a Riemannian metric as an analogue of an almost contact Riemannian manifold.
The study of the differential geometry of these manifolds began with \cite{AdatMiya77}, \cite{Sato77}, \cite{Sato78} by I.\,Sato, T.\,Adati and T.\,Miyazawa.
After that, in \cite{Sa80}, S. Sasaki defined the notion of an almost
paracontact Riemannian manifold of type $(p,q)$, where $p$ and $q$ are the multiples of the eigenvalues $+1$ and $-1$ of the paracontact endomorphism, respectively.
It also has a simple eigenvalue of 0.

In \cite{ManSta01}, M.\,Manev and M.\,Staikova gave a classification of almost paracontact Riemannian manifolds $(\MM,\f,\xi,\eta,g)$ of type $(n,n)$.
The dimension of $\MM$ is $2n+1$ and the induced almost product structure $P$ of $\f$
on the paracontact distribution $\ker(\eta)$ is traceless, \ie  $P$ is an almost paracomplex structure.
Because of this, the present authors called them
\emph{almost paracontact almost paracomplex Riemannian  manifolds} in  \cite{ManVes18}
and continued their study together with S.\,Ivanov and H.\,Manev (e.g. \cite{IvMan2, ManVes18, ManVes2}).

In the present paper, we study the geometry of almost paracontact almost paracomplex Riemannian manifolds of the lowest dimension 3. In Section~\ref{sect-prel}, we recall some necessary facts about the studied manifolds.
In Section~\ref{sect-C} and Section~\ref{sect-S}, we use two different approaches to construct a manifold of the studied type as a
product of a real line and a 2-dimensional manifold. 
In the first case, $g$ is the cone metric, and in a second one, $g$
is the so-called hyperbolic extension, introduced in \cite{IvMan2}, where the underlying manifold is paraholomorphic paracomplex Riemannian.
Our goal in the present work is to study the basic curvature properties  of the
resulting manifolds.


\section{Preliminaries}\label{sect-prel}

\subsection{Almost paracontact almost paracomplex Riemannian manifolds}


Let $(\MM,\f,\allowbreak{}\xi,\eta, g)$ be an \emph{almost paracontact almost paracomplex Riemannian manifold} (abbr. \emph{apapR manifold}).
This means that $\MM$ is a  $(2n+1)$-dimensional differentiable manifold,
$\f$ is a paracontact endomorphism   of the tangent bundle $T\MM$, $\xi$ is a characteristic vector field and  $\eta$ is its dual 1-form, and $g$ is a compatible Riemannian metric, such that:
\begin{equation}\label{str}
\begin{array}{c}
\f^2 = \I - \eta \otimes \xi,\quad \eta(\xi)=1,\quad
\eta\circ\f=0,\quad \f\xi = 0,\quad \tr\f = 0,\\[4pt]
g(\f x, \f y) = g(x,y) - \eta(x)\eta(y),
\end{array}
\end{equation}
where $\I$ denotes the identity on $T\MM$ \cite{Sato76}, \cite{ManSta01}.

Here and further $x$, $y$, $z$, $w$ will stand for arbitrary
elements of the Lie algebra $\X(\MM)$ of tangent vector fields on $\MM$ or vectors in the tangent space $T_p\MM$ at $p\in \MM$.

The associated metric $\tg$ of $g$ on $\M$ is determined  by
$\tg(x,y)=g(x,\f y)+\eta(x)\eta(y)$. In \cite{ManSta01}, it is shown  that $\tg$ is a compatible metric
with $(\MM,\f,\xi,\eta)$ as $g$, but $\tg$ is a pseudo-Riemannian metric of signature $(n + 1,
n)$.

The fundamental  tensor $F$ of type  $(0,3)$ on $(\MM,\f,\xi,\eta,g)$ is  defined by
\begin{equation*}\label{F=nfi}
F(x,y,z)=g\bigl( \left( \nabla_x \f \right)y,z\bigr),
\end{equation*}
where $\nabla $ is the Levi-Civita connection of $g$.
The basic properties of $F$ with respect to the structure are the following:
\begin{equation*}\label{F-prop}
\begin{array}{l}
F(x,y,z)=F(x,z,y)\\[4pt]
\phantom{F(x,y,z)}=-F(x,\f y,\f z)+\eta(y)F(x,\xi,z)
+\eta(z)F(x,y,\xi).
\end{array}
\end{equation*}

The Lee forms of $(\MM,\f,\xi,\eta,g)$ are the following 1-forms
 associated with $F$:
 \begin{equation}\label{t}
\theta(z)=g^{ij}F(e_i,e_j,z),\quad
\theta^*(z)=g^{ij}F(e_i,\f e_j,z), \quad \omega(z)=F(\xi,\xi,z),
\end{equation}
where $g^{ij}$  are the components of the inverse
matrix of $g$ with respect to a basis
$\left\{\xi;e_i\right\}$ $(i=1,2,\dots,2n)$ of $T_p\MM$ at an arbitrary point $p\in \MM$.

A  classification of apapR manifolds is made in \cite{ManSta01}.
It consists of 11 basic
classes $\F_1$, $\F_2$, $\dots$, $\F_{11}$ and it is made with respect to
 $F$.
The components $F_{i}$ of $F$ corresponding to $\F_{i}$ are determined in \cite{ManVes18}.
Namely,  $\M$ belongs to $\F_{i}$ $(i\in\{1,2,\dots,11\})$
if and only if the equality $F=F_i$ is valid.
Moreover,
$\M\in\F_i\oplus\F_j\oplus\cdots$ if and only if
the following condition is satisfied $F=F_i+F_j+\cdots$.

In the present  work, we consider the case of the lowest dimension of  $(\MM,\f,\xi,\eta,g)$, i.e. $\dim{\MM}=3$.
Let
$\{e_0, e_1, e_2\}$ be a $\f$-basis of $T_p\MM$, which satisfies the following
conditions:
\begin{equation}\label{fbasis}
\begin{array}{ll}
\f e_0=0,\quad \f e_1=e_{2},\quad \f e_{2}= e_1,\quad &\xi=
e_0,\quad \\[4pt]
\eta(e_0)=1,\quad \eta(e_1)=\eta(e_{2})=0,\qquad &g(e_i,e_j)=\delta_{ij},\; i,j\in\{0,1,2\}.
\end{array}
\end{equation}

The components  of $F$, $\ta$, $\ta^*$, $\om$ with respect to $\left\{e_0,e_1,e_2\right\}$ are denoted by ${F_{ijk}=F(e_i,e_j,e_k)}$, ${\ta_k=\ta(e_k)}$, ${\ta^*_k=\ta^*(e_k)}$ and ${\om_k=\om(e_k)}$, respectively.
According to \cite{ManVes18}, we have the following:
\begin{equation}\label{t3}
\begin{array}{c}
	\begin{array}{ll}
		\ta_0=F_{110}+F_{220},\quad & \ta_1=F_{111}=-F_{122}=-\ta^*_2,\\[4pt]
		\ta^*_0=F_{120}+F_{210}, \quad &\ta_2=F_{222}=-F_{211}=-\ta^*_1,\\[4pt]
	\end{array}\\
	\begin{array}{lll}
		\om_0=0,  \qquad & \om_1=F_{001},\qquad & \om_2=F_{002}.
	\end{array}
\end{array}
\end{equation}
Then, if   $F^s$, $s\in\{1,2,\dots, 11\}$, are the components  of $F$ 
in  the corresponding basic classes $\F_s$  and  $x=x^ie_i$, $y=y^ie_i$, $z=z^ie_i$ are  arbitrary vectors in
$T_p\MM$, we  have the following: \cite{ManVes18}
\begin{equation}\label{Fi3}
\begin{array}{l}
F^{1}(x,y,z)=\left(x^1\ta_1-x^2\ta_2\right)\left(y^1z^1-y^2z^2\right); \\[4pt]
F^{2}(x,y,z)=F^{3}(x,y,z)=0;
\\
F^{4}(x,y,z)=\frac{1}{2}\ta_0\bigl\{x^1\left(y^0z^1+y^1z^0\right)
+x^2\left(y^0z^2+y^2z^0\right)\bigr\};\\[4pt]
F^{5}(x,y,z)=\frac{1}{2}\ta^*_0\bigl\{x^1\left(y^0z^2+y^2z^0\right)
+x^2\left(y^0z^1+y^1z^0\right)\bigr\};\\[4pt]
F^{6}(x,y,z)=F^{7}(x,y,z)=0;\\[4pt]
F^{8}(x,y,z)=\lm\bigl\{x^1\left(y^0z^1+y^1z^0\right)
-x^2\left(y^0z^2+y^2z^0\right)\bigr\},\\[4pt]
\hspace{38pt} \lm=F_{110}=-F_{220}
;\\[4pt]
F^{9}(x,y,z)=\mu\bigl\{x^1\left(y^0z^2+y^2z^0\right)
-x^2\left(y^0z^1+y^1z^0\right)\bigr\},\\[4pt]
\hspace{38pt} \mu=F_{120}=-F_{210}
;\\[4pt]
F^{10}(x,y,z)=\nu x^0\left(y^1z^1-y^2z^2\right),\quad
\nu=F_{011}=-F_{022}
;\\[4pt]
F^{11}(x,y,z)=x^0\bigl\{\om_{1}\left(y^0z^1+y^1z^0\right)
+\om_{2}\left(y^0z^2+y^2z^0\right)\bigr\}.
\end{array}
\end{equation}

By virtue of the latter equations,  the 3-dimensional manifolds of the considered type can belong only to
the basic classes
$\F_1$, $ \F_4$, $\F_5$,  $\F_8$, $\F_9$, $\F_{10}$,  $\F_{11}$ and their direct sums \cite{ManVes18}.


The curvature $(1,3)$-tensor $R$ of $\nabla$ is defined as usually by $R=\left[\n,\n\right]-\n_{[\,\ ]}$.
The corresponding $(0,4)$-tensor is denoted by the same letter and it is determined  by  $R(x,y,z,w)=g(R(x,y)z,w)$.

The Ricci tensor $\rho$ and the scalar curvature $\tau$ for $R$ as well as
their associated quantities are defined respectively by:
\begin{equation}\label{rhotau}
\begin{array}{ll}
    \rho(y,z)=g^{ij}R(e_i,y,z,e_j),\qquad &
    \tau=g^{ij}\rho(e_i,e_j),\\[4pt]
    \rho^*(y,z)=g^{ij}R(e_i,y,z,\f e_j),\qquad &
    \tau^*=g^{ij}\rho^*(e_i,e_j).
\end{array}
\end{equation}

The following tensors are essential curvature-like tensors of type $(1,3)$:
\begin{equation}\label{pi12}
\begin{array}{l}
    \pi_1(x,y)z=g(y,z)x-g(x,z)y,\\[4pt]
    \pi_2(x,y)z=g(y,\f z)\f x-g(x, \f z)\f y. 
\end{array}
\end{equation}
Their corresponding curvature-like $(0,4)$-tensors are determined by $g$ as usually, $\pi_i(x,y,z,w)=g\left(\pi_i(x,y)z,w\right)$, $i=1,2$.


Let $\al$ be a non-degenerate 2-plane in $T_p\MM$, $p \in \MM$, having a basis  $\{x,y\}$.
The sectional curvature $k(\al;p)$  is determined by
\begin{equation}\label{sect}
k(\al;p)=\frac{R(x,y,y,x)}{\pi_1(x,y,y,x)}.
\end{equation}



\subsection{Almost paracomplex  Riemannian manifold}


Consider a differentiable
manifold $\NN$ of arbitrary dimension.
Let us recall that an \emph{almost product structure} $P$ on $\NN$
is an endomorphism in the tangent bundle $T\NN$ of $\NN$ such that
$P^2$ is the identity $I$ in $T\NN$, but $P$ does not coincide with $I$.
Then, such a manifold $(\NN, P)$ is called an \emph{almost product manifold}.
In the particular case when the eigenvalues $+1$ and $-1$ of $P$ have the same
multiplicity $n$, the structure $P$ is called an \emph{almost paracomplex structure}
and $(\NN, P)$ is known as an \emph{almost paracomplex manifold} of dimension $2n$
\cite{CrFoGa96}. In this case $P$ is traceless, \ie $\tr P=0$.
This kind of manifolds are also known as almost product Riemannian manifolds
with $\tr P=0$ in \cite{StaGri}.

As it is known, the $2n$-dimensional paracontact distribution
$\HH=\ker(\eta)$ of $(\MM,\f,\xi,\eta,g)$ can be considered as an almost paracomplex manifold {$\NN$}
equipped with an almost paracomplex structure
$P=\f|_\HH$ and a metric $h=g|_\HH$, where $\f|_\HH$ and $g|_\HH$ are  the
restrictions of $\f$ and $g$ on $\HH$, respectively.

Let $x'$, $y'$, $z'$, $w'$ denote arbitrary vector fields or vectors on $\HH$ {of $\MM$}.

Since $g$ is a Riemannian  metric of $(\MM,\f,\xi,\eta)$,
then $h$ is the corresponding Riemannian metric on $\HH$ and due to \eqref{str} it
is compatible with $P$ as follows
\begin{equation}\label{P-comp}
h(Px',Py')=h(x',y').
\end{equation}
The associated Riemannian  metric $\widetilde{h}$ of $h$ is determined by
\begin{equation*}\label{assoc-norden}
\widetilde{h}(x',y')=h(x',Py')
\end{equation*}
and it is a pseudo-Riemannian of signature $(n,n)$.

Let us note that an $2n$-dimensional manifold $\NN$, which is equipped with an almost paracomplex structure $P$ and a Riemannian metric $h$ satisfying \eqref{P-comp},
is known as an \emph{almost paracomplex Riemannian manifold} $(\NN,P,h)$.

In \cite{Nav}, A.\,M.\  Naveira gave a classification of almost product Riemannian
manifolds $(\NN, P, h)$ with respect
to the covariant derivative $\n' P$ for the Levi-Civita connection $\nabla'$ of $h$.

Furthermore, using this  classification,
M. Staikova and K. Gribachev present a classification of almost paracomplex Riemannian manifolds $(\NN, P, h)$ in \cite{StaGri}.
The basic classes of the Staikova-Gribachev classification are three, $\W_1$, $\W_2$ and $\W_3$.
Their intersection is the class $\W_0$
determined by the condition $\n' P=0$.
The manifolds of the latter class are known as
\emph{locally product Riemannian manifolds} \cite{Nav}, \emph{Riemannian $P$-manifolds} \cite{StaGri}
or \emph{paraholomorphic paracomplex Riemannian
manifolds} \cite{IvMan2}.
In the present paper, we call such manifolds briefly $\W_0$-manifolds.

Let us remark that $\W_1$ contains
the manifolds which are locally conformal equivalent to $\W_0$-manifolds.

In the present paper we deal with the case of almost paracomplex  Riemannian manifolds $(\NN,P,h)$ of the lowest dimension, \ie $\dim\NN=2$.
As it is known such a manifold is a space-form, \ie
the manifold has a curvature tensor of the form $R'=k'\,\pi'_1$,
where $k'$ is its pointwise constant sectional curvature,
and $\pi'_1$ is the essential curvature-like tensor, such as $\pi_1$ in \eqref{pi12}, but with respect to $h$.
Moreover, each $(\NN,P,h)$ is a $\W_1$-manifold determined by 
\begin{equation}\label{FW1}
\begin{array}{l}
      F'(x',y',z')=\frac12\{h(x',y')\ta'(z')+h(x',z')\ta'(y')\\[4pt]
  \phantom{F'(x',y',z')=\frac12\{}+h(x',Py')\ta'^*(z')+h(x',Pz')\ta'^*(y')\},
\end{array}
\end{equation}
where $F'$ is the fundamental tensor of $(\NN,P,h)$ defined by the following equality
$F'(x',y',z')=h\bigl( \left( \n'_{x'} P \right)y',z'\bigr)$, $\ta'$ is the Lee form determined analogously as in \eqref{t} 
 and $\ta'^*=-\ta'\circ P$ is valid \cite{StaGri}.

\section{Cone over a 2-dimensional paracomplex Riemannian space-form}\label{sect-C}

Let us consider $\mathcal{C}(\NN)=\R^+\times \NN$,
the cone over a 2-dimensional paracomplex space-form $(\NN,P,h)$,
where $\R^+$ is the set of positive reals. We introduce  a Riemannian metric $g$ on $\mathcal{C}(\NN)$  defined by
\begin{equation}\label{C-g}
 g\left(\left(x',a\ddt\right),\left(y',b\ddt\right)\right)
=t^2\,h(x',y')+ab,
\end{equation}
where $t$ is the coordinate on $\R^+$ and $a$, $b$ are
differentiable functions on $\mathcal{C}(\NN)$.

We equip $\mathcal{C}(\NN)$ with  an almost paracontact almost paracomplex structure $(\f,\xi,\eta)$
by the following way
\begin{equation}\label{C-str}
\f |_\HH=P, \quad \xi=\ddt, \quad \eta=\D t, \quad \f\xi=0,\quad \eta\circ\f =0.
\end{equation}
Obviously, we establish the  truthfulness of the following
\begin{prop}
The manifold $(\mathcal{C}(\NN),\f,\xi,\eta,g)$ is a 3-dimensional  apapR manifold.
\end{prop}

According to \eqref{C-g} and \eqref{C-str} and the well known Koszul equality
\begin{equation}\label{koszul}
\begin{array}{l}
2g(\n_{x}y,z)=xg(y,z)+yg(z,x)-zg(x,y)\\[4pt]
\phantom{2g(\n_xy,z)=}+g([x,y],z)+g([z,x],y)+g([z,y],x),
\end{array}
\end{equation}
we get the following equalities for
the Levi-Civita connection $\n$ of $g$:
\begin{equation*}\label{C-n-g}
\begin{array}{ll}
    g\left(\n_{x'} y',z'\right)=t^2\,h\left(\n'_{x'} y',z'\right),\quad &
    g\left(\n_\xi y',z'\right)=t\,h\left(y', z'\right), %
    \\[4pt]
    g\left(\n_{x'} y',\xi\right)=-t\,h\left(x',y'\right),\quad &
    g\left(\n_{x'} \xi,z'\right)=t\,h\left(x', z'\right).
\end{array}
\end{equation*}

Bearing in mind the latter equalities,  we obtain  the covariant derivatives with respect to $\n$ as follows
\begin{equation}\label{C-n}
    \n_{x'} y'=\n'_{x'} y'-\frac{1}{t}g\left(x', y'\right)\xi,\quad
    \n_\xi y'=\frac{1}{t}y',\quad
    \n_{x'} \xi=\frac{1}{t}x'.
\end{equation}

By virtue of \eqref{C-n}, we get
\begin{equation*}\label{C-Rxyz}
\begin{array}{l}
    R(x',y')z'=\frac{1}{t^2}(k'-1)\pi_1(x',y')z',\\[4pt]
    R(x',y')\xi=R(x',\xi)y'=R(\xi,x')y'=R(x',\xi)\xi=R(\xi,y')\xi=0.
\end{array}
\end{equation*}
Thus, by direct computations, we establish the following
\begin{equation}\label{C-Rxyzw}
\begin{array}{l}
    R(x',y',z',w')=\frac{1}{t^2}(k'-1)\pi_1(x',y',z',w'),\\[4pt]
    R(\xi,x',y',z')=R(x',\xi,y',z')=R(x',y',\xi,z')=R(x',y',z',\xi)\\[4pt]
    \phantom{R(\xi,x',y',z')}
    =R(y',\xi,\xi, z')=R(\xi,y',z',\xi)=0.
\end{array}
\end{equation}

Using  \eqref{fbasis} and \eqref{C-g}, we determine the components $h_{ij}=h(e_i,e_j)$ and $g_{ij}=g(e_i,e_j)$ with respect to the basis. The non-zero ones of them are the following:
\begin{equation}\label{C-hijgij}
h_{11}=h_{22}=\frac{1}{t^2}, \quad g_{00}=g_{11}=g_{22}=1.
\end{equation}

Taking into account \eqref{koszul}, \eqref{C-n} and \eqref{C-hijgij}, we determine the components of the covariant derivatives of $e_i$ with respect to $\n$:
\begin{equation}\label{C-nnij}
\begin{array}{lll}
\n_{e_1}e_1=\n'_{e_1}e_1-\frac{1}{t}e_0, \quad &\n_{e_1}e_2=\n'_{e_1}e_2, \quad &\n_{e_1}e_0=\frac{1}{t}e_1,\\[4pt]
\n_{e_2}e_1=\n'_{e_2}e_1, \quad &\n_{e_2}e_2=\n'_{e_2}e_2-\frac{1}{t}e_0, \quad &\n_{e_2}e_0=\frac{1}{t}e_2,\\[4pt]
\n_{e_0}e_1=\frac{1}{t}e_1, \quad &\n_{e_0}e_2=\frac{1}{t}e_2,  \quad &\n_{e_0}e_0=0.
\end{array}
\end{equation}

Bearing in mind \eqref{C-str}, \eqref{C-hijgij} and \eqref{C-nnij}, we obtain the components $F_{ijk}$
of $F$ with respect to the basis  $\{e_0,e_1,e_2\}$. The non-zero ones of them are:
\begin{equation}\label{C-Fijk}
\begin{array}{l}
  F_{111}=-F_{122}=\ta'_1, \qquad  F_{222}=-F_{211}=\ta'_2,\\[4pt]
  F_{120}=F_{102}=F_{210}=F_{201}=-\frac{1}{t},
\end{array}
\end{equation}
where $\ta'_i=\ta'(e_i)$ for $i=1,2$.

According to \eqref{t3}, we compute the components of the Lee forms of $(\mathcal{C}(\NN), \f, \xi,\allowbreak{} \eta, g)$ and the non-zero ones of them are:
\begin{equation}\label{ta-12}
\ta_1=\ta'_1, \quad \ta_2=\ta'_2, \quad \ta^*_1=-\ta'_2, \quad \ta^*_2=-\ta'_1, \quad \ta^*_0=-\frac{2}{t}.
\end{equation}

By virtue of \eqref{Fi3} and \eqref{C-Fijk}, we establish the following form of the tensor $F$:
\begin{equation*}\label{C-F1+F5}
F(x,y,z)=(F^1+F^5)(x,y,z).
\end{equation*}
According to \eqref{Fi3} and \eqref{FW1},  the non-zero components of $F^1$ and $F^5$  are the following:
\begin{equation}\label{Fijk15}
\begin{array}{l}
F^1_{111}=-F^1_{122}=\ta_1, \quad F^1_{222}=-F^1_{211}=\ta_2,\\[4pt]
F^5_{120}=F^5_{102}=F^5_{210}=F^5_{201}=\frac{1}{2}\ta^*_0.
\end{array}
\end{equation}
Thus, we obtain the  following
\begin{thm}
The 3-dimensional apapR manifold $(\mathcal{C}(\NN), \f, \xi, \eta, g)$:
\begin{enumerate}
  \item belongs to $\F_1\oplus\F_5$,
  \item belongs to $\F_5$ if and only if $(\NN,P,h)$ is a $\W_0$-manifold,
  \item cannot belong to $\F_1$.
\end{enumerate}
\end{thm}
\begin{proof}
We compare \eqref{C-Fijk}, \eqref{ta-12} and \eqref{Fijk15} to accomplish the proof.
\end{proof}


Next, bearing in mind  \eqref{C-Rxyzw}, \eqref{C-hijgij} and \eqref{C-nnij}, we determine   the basic  components $R_{ijk\ell}=R(e_i,e_j,e_k,e_\ell)$ of $R$. The non-zero ones of them are obtained by the basic symmetries of $R$ and the following
\begin{equation}\label{C-Rijkl}
  R_{1212}=-\frac{1}{t^2}(k'-1).
\end{equation}

Thus, it is valid the following

\begin{thm}
The 3-dimensional apapR manifold $(\mathcal{C}(\NN), \f, \xi, \eta, g)$ is flat if and only if $k'=1$.
\end{thm}

Moreover, using \eqref{sect}, \eqref{C-hijgij} and \eqref{C-Rijkl}, we compute the basic sectional curvatures $k_{ij}=k(e_i,e_j)$ determined by the basis $\{e_i, e_j\}$ of the corresponding 2-plane as
follows
\begin{equation}\label{C-kij}
k_{12}=\frac{1}{t^2}(k'-1),\qquad k_{01}=k_{02}=0.
\end{equation}

By virtue of \eqref{rhotau}, \eqref{C-hijgij} and \eqref{C-Rijkl}, we get the basic components $\rho_{jk}=\rho(e_j,e_k)$ and $\rho^*_{jk}=\rho^*(e_j,e_k)$ of  $\rho$ and $\rho^*$, respectively, as well as the values
of $\tau$ and $\tau^*$. The non-zero ones of them are the following:
\begin{equation}\label{C-rhotau}
\begin{array}{lll}
\rho_{11}=\rho_{22}=-\rho^{*}_{12}=-\rho^{*}_{21}=\frac12\tau=\frac{1}{t^2}(k'-1).
\end{array}
\end{equation}

Taking into account \eqref{C-kij} and \eqref{C-rhotau}, we get the following
\begin{thm}
 The following properties of $(\mathcal{C}(\NN), \f, \xi, \eta, g)$ are valid:
\begin{enumerate}
  \item its sectional curvatures of $\xi$-sections vanish;
  \item it is $*$-scalar flat, \ie $\tau^*=0$;
  \item $\tau<0$ if and only if $k'<1$;
  \item $\tau>0$ if and only if $k'>1$.
\end{enumerate}
\end{thm}

\section{Hyperbolic extension of a 2-dimensional paracomplex Riemannian space-form}\label{sect-S}

In this  section, we construct a special type of a 3-dimensional warped product manifold
$\mathcal{S}(\NN)$ of $\mathbb R^+$ and a paracomplex space-form $(\NN,P,h)$ from the class $\W_1$.
Then, $(\NN,P,h)$ is defined by \eqref{FW1} for some 1-form $\ta'$.

Let $\D t$ be the coordinate 1-form on $\mathbb R^+$ and let us introduce an
almost paracontact almost paracomplex structure and a Riemannian metric  on $\mathcal{S}(\NN)$ as follows
\begin{equation}\label{S-str}
\f |_\HH=P, \quad \xi=\ddt, \quad \eta=\D t, \quad \eta\circ\f =0,\quad g=\D t^2+\cosh{2t}\,h+\sinh{2t}\,\widetilde{h}.
\end{equation}

Then, it is easy to check the following
\begin{prop}
The manifold $(\mathcal{S}(\NN),\f,\xi,\eta,g)$ is a 3-dimensional apapR manifold.
\end{prop}

In the partial case, when $(\NN,P,h)$ is paraholomorphic, \ie $\n'P=0$, then the Riemannian manifold $(\mathcal{S}(\NN),\f,\xi,\eta,g)$ is para-Sasaki-like \cite{IvMan2}.

Taking into account \eqref{FW1}, \eqref{koszul} and \eqref{S-str}, we get the following formulae for the Levi-Civita connections $\n$ and $\n'$ of $g$ and $h$, respectively:
\begin{equation}\label{S-nn}
\begin{array}{c}
    \n_{x'} y'=\n'_{x'} y'+ \frac{1}{2}\sinh{2t}\left\{g\left(x',y'\right)\ta'^{\sharp}- g\left(x',Py'\right)P\ta'^{\sharp}\right\}-g\left(x',Py'\right)\xi,\\[4pt]
    \n_\xi y'=Py',\qquad
    \n_{x'} \xi=Px',\qquad \n_{\xi} \xi=0.
\end{array}
\end{equation}
 In the latter equalities and further, $\ta'^{\sharp}$ denotes the dual vector of 
 $\ta'$ with respect to $h$
 on $(\NN,P,h)$.
 Analogously, $\ta^{\sharp}$ stands for the dual vector of 
 $\ta$ with respect to $g$ on $(\mathcal{S}(\NN),\f,\xi,\eta,g)$.
Bearing in mind \eqref{S-str}, we have the following relations
\begin{equation*}\label{ta-ta'}
    \ta^\s\vert_\HH = \cosh 2t\, \ta'^\s - \sinh 2t\, P\ta'^\s,\qquad \ta(z')=\ta'(z').
\end{equation*}

Taking into account \eqref{S-nn}, we compute the following
\begin{equation}\label{S-Rxyz}
\begin{array}{l}
    R(x',y')z'=k'\left\{\cosh2t\,\pi_1(x',y')z' - \sinh2t\,P\pi_2(x',y')z'\right\} - \pi_2(x',y')z'\\[4pt]
        \phantom{R(x',y')z'=}
    +
    \frac{1}{2}\sinh 2t\left[
    g\left(y',z'\right)\n_{x'}\ta'^\s-g\left(x',z'\right)\n_{y'}\ta'^\s\right.\\[4pt]
    \phantom{R(x',y')z'=+\frac{1}{4}\sinh 2t\left[\right.}
    \left.
    -g\left(y',Pz'\right)\n_{x'}P\ta'^{\sharp}+g\left(x',Pz'\right)\n_{y'}P\ta'^{\sharp}\right]\\[4pt]
    \phantom{R(x',y')z'=}
    +\left.\frac{1}{4} \{\pi_1-\pi_2\}(x',y',z',\ta'^{\sharp})
    \left\{\sinh 2t\, P\ta^{\sharp}\vert_\HH+2\cosh 2t\, \xi\right\} \right.\\[4pt]
    R(x',y')\xi=
    \frac{1}{2}\{\pi_1-\pi_2\}(x',y')\ta^{\sharp}\vert_\HH,\\[4pt]
    R(\xi,y')z'=
    \frac{1}{2}\{\pi_1-\pi_2\}(\ta^{\sharp}\vert_\HH,y')z'-g\left(y',z'\right)\xi, \\[4pt]
    R(x',\xi)z'=
    \frac{1}{2}\{\pi_1-\pi_2\}(x',\ta^{\sharp}\vert_\HH)z'+g\left(x',z'\right)\xi
 .
\end{array}
\end{equation}

Bearing in mind   \eqref{S-str} and \eqref{S-Rxyz}, we obtain the following

\begin{equation}\label{S-Rxyzw}
\begin{array}{l}
   R(x',y',z',w')=k'\{\cosh2t\,\pi_1(x',y',z',w')-\sinh2t\,\pi_2(x',y',z',Pw')\}\\[4pt]
    \phantom{R(x',y',z',w')=}-\pi_2(x',y',z',w')\\[4pt]
    \phantom{R(x',y',z',w')=}
    +
    \frac{1}{2}\sinh 2t\left\{
    g\left(y',z'\right)g\left(\n_{x'}\ta'^\s,w'\right)\right.\\[4pt]
    \phantom{R(x',y',z',w')=+\frac{1}{2}\sinh 2t\left\{\right.}
    -g\left(x',z'\right)g\left(\n_{y'}\ta'^\s,w'\right)\\[4pt]
    \phantom{R(x',y',z',w')=+\frac{1}{2}\sinh 2t\left\{\right.}
    -g\left(y',Pz'\right)g\left(\n_{x'}P\ta'^{\sharp},w'\right)\\[4pt]
    \phantom{R(x',y',z',w')=+\frac{1}{2}\sinh 2t\left\{\right.}
    \left.
    +g\left(x',Pz'\right)g\left(\n_{y'}P\ta'^{\sharp},w'\right)\right\}\\[4pt]
    \phantom{R(x',y',z',w')=}
    +\left.\frac{1}{4}\sinh 2t\, \{\pi_1-\pi_2\}(x',y',z',\ta'^{\sharp})
    \ta'(Pw'), \right.\\[4pt]
    R(x',y',z',\xi)=\frac{1}{2} \{\pi_1-\pi_2\}(x',y',z',\ta^\s\vert_\HH),\\[4pt]
    R(x',y',\xi,w')=\frac{1}{2} \{\pi_1-\pi_2\}(x',y',\ta^{\sharp}\vert_\HH,w'),\\[4pt]
    R(x',\xi,z',w')=\frac{1}{2} \{\pi_1-\pi_2\}(x',\ta^{\sharp}\vert_\HH,z',w'),\\[4pt]
    R(\xi,y',z',w')=\frac{1}{2} \{\pi_1-\pi_2\}(\ta^{\sharp}\vert_\HH,y',z',w'),\\[4pt]
    R(\xi,y',z',\xi)=-g\left(y',z'\right).
\end{array}
\end{equation}

According to \eqref{fbasis} and \eqref{S-str}, we determine the components $g_{ij}$ and $h_{ij}$ 
 as follows
\begin{equation}\label{S-hijgij}
\begin{array}{ll}
g_{00}=g_{11}=g_{22}=1, \quad &g_{12}=g_{21}=0,\\[4pt]
h_{11}=h_{22}=\cosh2t, \quad &h_{12}=h_{21}=-\sinh2t.
\end{array}
\end{equation}

Bearing in mind \eqref{koszul}, \eqref{S-nn} and \eqref{S-hijgij}, we obtain
\begin{equation}\label{S-nnij}
\begin{array}{ll}
\n_{e_1}e_1=\n'_{e_1}e_1+\frac{1}{2}\sinh2t\,\ta'^{\sharp}, \qquad  \n_{e_1}e_2=\n'_{e_1}e_2-\frac{1}{2}\sinh2t\,P\ta'^{\sharp}-\xi,\\[4pt]
\n_{e_2}e_2=\n'_{e_2}e_2+\frac{1}{2}\sinh2t\,\ta'^{\sharp}, \qquad
\n_{e_2}e_1=\n'_{e_2}e_1-\frac{1}{2}\sinh2t\,P\ta'^{\sharp}-\xi, \\[4pt]
\n_{e_0}e_1=\n_{e_1}e_0=e_2, \qquad  \n_{e_0}e_2=\n_{e_2}e_0=e_1,  \qquad \n_{e_0}e_0=0.
\end{array}
\end{equation}

Then, applying \eqref{S-str}, \eqref{S-hijgij} and \eqref{S-nnij}, we determine the components $F_{ijk}$ of $F$. The non-zero of them are the following
\begin{equation}\label{S-Fijk}
\begin{array}{l}
  F_{111}=-F_{122}=\ta'_1, \qquad  F_{222}=-F_{211}=\ta'_2,\\[4pt]
  F_{101}=F_{110}=F_{202}=F_{220}=-1.
\end{array}
\end{equation}

According to \eqref{t3}, we get the components of the Lee forms of $(\mathcal{S}(\NN), \f, \xi,\allowbreak{} \eta, g)$ with respect to the basis. The non-zero  of them are:
\begin{equation}\label{ta-14}
\ta_0=-2, \quad \ta_1=-\ta^*_2=\ta'_1, \quad \ta_2=-\ta^*_1=\ta'_2.
\end{equation}

By virtue of  \eqref{Fi3} and \eqref{S-Fijk}, we establish the truthfulness of the following  equality
\begin{equation*}\label{S-F1+F4}
F(x,y,z)=(F^1+F^4)(x,y,z).
\end{equation*}
By virtue of \eqref{Fi3} and \eqref{FW1}, we have the following non-zero components of $F^1$ and $F^4$ with respect to the basis \eqref{fbasis}:
\begin{equation}\label{Fijk14}
\begin{array}{l}
F^1_{111}=-F^1_{122}=\ta_1, \quad F^1_{222}=-F^1_{211}=\ta_2,\\[4pt]
F^4_{101}=F^4_{110}=F^4_{202}=F^4_{212}=\frac{1}{2}\ta_0.
\end{array}
\end{equation}
Thus, we get the following
\begin{thm}\label{thm14}
The 3-dimensional apapR manifold $(\mathcal{S}(\NN), \f, \xi, \eta, g)$:
\begin{enumerate}
  \item belongs to $\F_1\oplus\F_4$,
  \item belongs to $\F_4$ if and only if $(\NN,P,h)$ is a $\W_0$-manifold,
  \item cannot belong to $\F_1$.
\end{enumerate}
\end{thm}
\begin{proof}
It follows from the above by comparing  \eqref{S-Fijk}, \eqref{ta-14} and \eqref{Fijk14}.
\end{proof}

Taking into account  \eqref{S-Rxyzw}, \eqref{S-hijgij} and \eqref{S-nnij}, we obtain the components $R_{ijk\ell}$ of $R$ with respect to the basis. The non-zero of them are determined by the basic symmetries of $ R$ and the following equalities
\begin{equation}\label{S-Rijkl}
\begin{array}{l}
  R_{1221}=k'\cosh2t+1\\[4pt]
\phantom{R_{1221}=}  +\frac{1}{2}\sinh2t\,
\left\{g\left(\n_{e_1}\ta'^\s, e_1\right)+g\left(\n_{e_2}P\ta'^{\s},e_1\right)\right\}+\frac12\ta'_1\ta'_2,\\[4pt]
  R_{1210}=-\ta'_2, \quad R_{1220}=\ta'_1, \quad R_{0110}=R_{0220}=-1.
\end{array}
\end{equation}

Bu virtue of  \eqref{sect}, \eqref{S-hijgij} and \eqref{S-Rijkl}, we calculate the basic sectional curvatures $k_{ij}$ as follows
\begin{equation}\label{S-kij}
k_{12}=R_{1221},\qquad k_{01}=k_{02}=-1.
\end{equation}

Furthermore, from \eqref{rhotau}, \eqref{S-hijgij} and \eqref{S-Rijkl}, we obtain  the basic components $\rho_{jk}$ and $\rho^*_{jk}$ as well as the values $\tau$ and $\tau^*$ as follows:
\begin{equation}\label{S-rhotau}
\begin{array}{lll}
\rho_{11}=\rho_{22}=R_{1221}-1, \quad & \rho_{00}=-2, \quad &
\rho^{*}_{00}=\rho^{*}_{11}=\rho^{*}_{22}=0, \\[4pt]
\rho_{12}=\rho_{21}=0, \quad & \rho_{01}=\rho_{10}=\ta'_1, \quad & \rho_{02}=\rho_{20}=\ta'_2\\[4pt]
\rho^{*}_{12}=\rho^{*}_{21}=-R_{1221}, \quad & \rho^{*}_{01}=\rho^{*}_{10}=-\ta'_2, \quad & \rho^{*}_{02}=\rho^{*}_{20}=-\ta'_1,\\[4pt]
\tau=2R_{1221}-4, \quad & \tau^{*}=0.
\end{array}
\end{equation}

Using  \eqref{S-kij} and \eqref{S-rhotau}, we conclude the following
\begin{prop}
The manifold $(\mathcal{S}(\NN), \f, \xi, \eta, g)$ has the following properties:
\begin{enumerate}
  \item It has constant negative $\xi$-sectional curvatures;
  \item It is $*$-scalar flat.
\end{enumerate}
\end{prop}

Bearing in mind  \eqref{S-rhotau}  and \eqref{S-hijgij} for $g_{ij}$, we obtain the following
\begin{thm}\label{thm4}
The following properties of $(\mathcal{S}(\NN), \f, \xi, \eta, g)$ are equivalent:
\begin{enumerate}
  \item $(\NN,P,h)$ is a $\W_0$-manifold;
  \item $\rho=k'\cosh2t\,g-(2+k'\cosh2t)\eta\otimes\eta$;
  \item $\rho^*=-(1+k'\cosh2t)(\widetilde{g}-\eta\otimes\eta)$.
\end{enumerate}
\end{thm}
Let us remark that (1) of the latter theorem is equivalent to the fact that $(\mathcal{S}(\NN), \f, \xi, \eta, g)$ is an $\F_4$-manifold, according to (2) of \thmref{thm14}.

The expression of the Ricci tensor in (2) of \thmref{thm4} means that, in this special case, $(\mathcal{S}(\NN), \f, \xi,\A \eta, g)$ is an para-$\eta$-Einstein manifold following \cite{ManVes2}.

Then, we obtain immediately the following

\begin{cor}
If $(\NN, P,h)$ is a $\W_0$-manifold, then $(\mathcal{S}(\NN), \allowbreak{}\f, \xi, \eta, g)$ has the properties:
\begin{enumerate}
  \item $\tau=2(k'\cosh2t-1)$;
  \item $k'<0$ if and only if $\tau\leq 2(k'-1)< -2$;
  \item $k'=0$ if and only if $\tau=-2$;
  \item $k'>0$ if and only if $-2 < 2(k'-1)\leq \tau $.
\end{enumerate}
\end{cor}

\bigskip

\subsection*{Acknowledgment}
The authors were supported by MU21-FMI-008 and the first author was partially supported by FP21-FMI-002, both projects
of the Scientific Research Fund,
University of Plovdiv Paisii Hilendarski, Bulgaria.

 \end{document}